\documentclass[a4paper, 11pt, reqno]{amsart}
\usepackage{amsmath,amsfonts,amssymb,amsthm,enumerate}
\usepackage{graphicx}
 \usepackage{xy} \xyoption{all}
 \usepackage{BOONDOX-frak}

 \usepackage{tikz}
\usetikzlibrary{arrows,shapes,snakes,automata,backgrounds}

\newcommand{\m}{\mathfrak}

\newtheorem{thm}{Theorem}[section]
\newtheorem{cor}[thm]{Corollary}
\newtheorem{prop}[thm]{Proposition}
\newtheorem{lem}[thm]{Lemma}

\theoremstyle{definition}
\newtheorem{defn}[thm]{Definition}
\newtheorem{exas}[thm]{Example}
\newtheorem{rem}[thm]{Remark}

\let\phi\varphi

\begin{document}
\title{On the homological dimensions of Leavitt path algebras with coefficients in
commutative rings}
\maketitle
\begin{center}
V.~Lopatkin\footnote{South China Normal University, Guangzhou, China. E-mail address: \texttt{wickktor@gmail.com}} and
T.\,G.~Nam\footnote{Institute of Mathematics, VAST, 18 Hoang Quoc Viet, Cau Giay, Hanoi, Vietnam. E-mail address: \texttt{tgnam@math.ac.vn}} 
\end{center}

\begin{abstract} In this paper, we give sharp bounds for the homological
dimensions of the Leavitt path algebra $L_K(E)$ of a finite graph $E$ with coefficients
in a commutative ring $K$, as well as establish a formula for calculating the homological
dimensions of $L_K(E)$ when $K$ is a commutative unital algebra over a field.
\medskip



\end{abstract}

\section{Introduction}
Throughout this note, all rings are nonzero, associative with identity and
all modules are unitary, unless otherwise stated. The categories of left modules
over a ring $R$ is denoted by ${R}\mbox{-}\mathrm{Mod}$.

Given a (row-finite) directed graph $E$ and a field $K$, Abrams and Aranda Pino in \cite{ap:tlpaoag05}, and independently Ara, Moreno, and Pardo in \cite{amp:nktfga}, introduced the \emph{Leavitt path algebra} $L_K(E)$. These Leavitt path algebras generalize the Leavitt algebras $L_K(1, n)$ of \cite{leav:tmtoar}, and also contain many other interesting classes of algebras. In addition, Leavitt path algebras are intimately related to graph $C^*$-algebras (see \cite{r:ga}). Later, in \cite{tomf:lpawciacr} Tomforde generalized the construction of Leavitt path algebras by replacing the field with a commutative ring.

In \cite[Theorem 3.5]{amp:nktfga} Ara, Moreno and Pardo established an important result that the global dimension of the Leavitt path algebra of a finite graph with coefficients a field is at most equal $1$, by using Bergman's nice machinery \cite{b:casurc}. Continuously of this research direction, in this paper, we establish sharp bounds of the homological dimensions of Leavitt path algebras of finite graphs with coefficients in commutative rings, as well as calculate exactly the homological dimensions of those Leavitt path algebras with coefficients in commutative unital algebras over fields. Here our method is different from Ara et. al.'s method, since Bergman's machinery can not work on algebras over commutative rings in general. As an application of the obtained result, we may cover Ara et. al.'s result cited above.

More precisely, the paper is organized as follows. In Section 2, we note that  the classical Composition--Diamond lemma for free associate algebras over fields has, in some cases, the same form for free associate algebras over an arbitrary ring (Theorem \ref{CDA}). For the reader's convenience, all subsequently necessary notions and notations on graph, Leavitt path algebras are included in Section 3. Also, in this section, as an application of Theorem \ref{CDA}, we express an analog of Zelmanov et. al.'s result \cite[Theorem 1]{aajz:lpaofgkd} for Leavitt path algebras with coefficients in commutative rings (Theorem \ref{B-LPA}). Applying the obtained result, we establish an analogous of  Abrams, Aranda Pino and Siles Molina's theorem \cite[Theorems 3.8 and 3.10]{apm:lflpa} for the structure of the Leavitt path algebra of a finite no-exit graph with coefficients in a commutative ring (Corollary \ref{cor3.5}).

In Section 4, we give sharp bounds for the homological dimensions of the Leavitt path algebra
of a finite graph with coefficients in a commutative ring, by using W. Dicks's interesting
result \cite[Corollary 7]{dicks:mvpocor} and Morita equivalence (Theorem \ref{thm4.7}).
Also, we provide a formula for calculating the homological dimensions of the Leavitt path
algebra of a finite graph with coefficients in a commutative unital algebra over a field (Theorem \ref{thm4.8}).


\section{Composition--Diamond lemma for associative algebras}

Here we present the concepts of Composition--Diamond lemma and Gr\"obner--Shirshov basis, see the survey \cite{BokSurv}. In the classical version of Composition--Diamond lemma, it assumed that considered algebras is over a field, here we consider the general case.

Let $K$ be an arbitrary commutative ring with unit, $K \langle X \rangle$ the free associative algebra over $K$ generated by $X$, and let $X^*$ be the free monoid generated by $X$, where empty word is the identity, denoted by $\mathbf{1}_{X^*}$.
Assume that $X^*$ is a well-ordered set. Take $f\in K \langle X \rangle$ with the leading word (term) $\overline{f}$ and $f= \kappa\overline{f} + r_{f}$, where $0\neq \kappa\in K$ and $\overline{r_f}<\overline{f}$. We call $f$ is
\emph{monic} if $\kappa = 1$. We denote by $\mathrm{deg}(f)$ the degree of $\overline{f}$.

\smallskip

\par A well ordering $\leqslant$ on $X^*$ is called {\it monomial} if for $u,v \in X^*$, we have:
\[
u\leqslant v \Longrightarrow \bigl.w\bigr|_u \leqslant \bigl.w\bigr|_v, \qquad \forall w \in X^*,
\]
where $\bigl.w\bigr|_u : = \bigl.w\bigr|_{x \to u}$ and $x$'s are the same individuality of the letter $x \in X$ in $w$.

A standard example of monomial ordering on $X^*$ is the deg-lex ordering (i.e., degree and lexicographical), in which two words are compared first by the degree and then lexicographically, where $X$ is a well-ordering set.

Fix a monomial ordering $\leqslant$ on $X^*$, and let $\varphi$ and $\psi$ be two monic polynomials in $K\langle X\rangle$. There are two kinds of compositions:
\begin{itemize}
\item[(i)] If $w$ is a word (i.e, it lies in $X^*$) such that $w = \overline{\varphi} b = a \overline{\psi}$ for some $a,b\in X^*$
with $\mathrm{deg}(\overline{\varphi})+\mathrm{deg}(\overline{\psi}) >\mathrm{deg}(w)$, then the polynomial $(\varphi,\psi)_w:=\varphi b - a \psi$ is called the {\it intersection composition} of $\varphi$ and $\psi$ with respect to $w$.
\item[(ii)] If $w = \overline{\varphi} = a\overline{\psi} b$ for some $a,b \in X^*$, then the polynomial $(\varphi,\psi)_w:=\varphi -a\psi b$ is called the {\it inclusion composition} of $\varphi$ and $\psi$ with respect to $w$.
\end{itemize}

We then note that $\overline{(\varphi,\psi)_w}\leq w$ and $(\varphi,\psi)_w$ lies in the ideal $(\varphi,\psi)$ of $K\langle X\rangle$ generated by $\varphi$ and $\psi$.

Let $\mathfrak{S}\subseteq K\langle X \rangle$ be a monic set (i.e., it is a set of monic polynomials). Take $f \in K\langle X\rangle$ and $w\in X^*$. We call $f$ is {\it trivial modulo} $(\mathfrak{S},w)$, denoted by \[f \equiv 0 \bmod(\mathfrak{S},w),\]
if $f = \sum\limits_{\mathfrak{s}\in \mathfrak{S}} \kappa a \mathfrak{s} b$, where $\kappa \in K$, $a,b \in X^*$, and $a \overline{\mathfrak{s}}b \leqslant w$.

A monic set $\mathfrak{S}\subseteq K\langle X \rangle$ is called a \emph{Gr\"obner--Shirshov basis} in $K\langle X \rangle$ with respect to the monomial ordering $\leq$ if every composition of polynomials in $\mathfrak{S}$ is trivial modulo $\mathfrak{S}$ and the corresponding $w$.

The following Composition--Diamond lemma was first proved by Shirshov \cite{Sh2} (see also the survey \cite{BokSurv}) for free Lie algebras over fields (with deg-lex ordering). For commutative algebras, this lemma is known as Buchberger's theorem \cite{Buchberger}.

\begin{thm}\label{CDA}
Let $K$ be an arbitrary commutative ring with unit, $\leqslant$ a monomial ordering on $X^*$
and let $I(\mathfrak{S})$ be the ideal of $K \langle X \rangle$ generated by the monic
set $\mathfrak{S}\subseteq K \langle X \rangle$. Then the following statements are equivalent:
\begin{itemize}
\item[(1)] $\mathfrak{S}$ is a Gr\"obner--Shirshov basis in $K \langle X \rangle$;
\item[(2)] the set of irreducible words
   \[\mathrm{Irr}(\mathfrak{S}):=\left\{u \in X^*: u \ne a \overline{\mathfrak{s}}b,\,\m{s} \in \mathfrak{S},\,a,b\in X^*\right\}\]
is a linear basis of the algebra $K \langle \bigl.X\bigr|\mathfrak{S}\rangle:=K \langle X\rangle/I(\mathfrak{S})$.
\end{itemize}
\end{thm}
\begin{proof}
The formal proof of the Composition--Diamond lemma for free associative algebras over field is found in \cite[Theorem 1]{BokSurv}. We essentially follow the ideas in this proof, and just for the reader's convenience, we briefly sketch it here.

\par $(1) \Longrightarrow (2).$ Let $\m{S}$ be a set of monic polynomials and let us assume that $\m{S}$ is a Gr\"obner--Shirshov basis. Using the same argument given in \cite[Theorem 1]{BokSurv}, we get that if $f\in I(\m{S})$, then $\overline{f} = a \overline{\m{s}}b$, where $a,b\in X^*$, $\m{s} \in \m{S}$. Using this observation, we obviously get for each $f\in K\langle X\rangle$,
\[
 I(\m{S}) \ni f \Longleftrightarrow f^{(1)}:=f -a{\m{s}}b \in I(\m{S}),
\]
for some $a, b\in X^*$ and $\m{s}\in \m{S}$. Continuing this line of reasoning, we see that
\[
 I(\m{S}) \ni f \Longleftrightarrow f^{(i)}: = f^{(i-1)}-\kappa^{(i-1)}a'\m{s}'b' \in I(\m{S}),
\]
for some $a', b'\in X^*$, $\m{s'}\in \m{S}$ and $\kappa^{(i-1)}\in K$ (note that $\kappa^{(i-1)}$ is the coefficient of the leading term of the polynomial $f^{(i-1)}$). We set $f^{(0)}:=f$. By using the process at finitely many times, we immediately obtain that
\[I(\m{S}) \ni f \Longleftrightarrow f^{(\ell_f)} = 0 \text{ for some positive integer } \ell_f.\]

Since $\m{S}$ is a Gr\"obner--Shirshov basis, then for any $\m{s}\in \m{S}$ there exists a positive integer $\ell_{\m{s}} \in \mathbb{N}$ such that $\m{s}^{(\ell_{\m{s}})}=0$.  From these notes, we immediately get that if $f \notin I(\m{S})$, then we may find a positive integer $t_f$ such that \[
f^{(t_f)} = \sum\kappa \widehat{f},
\]
where $0\neq\kappa \in K$, and $\widehat{f} \ne a\overline{\m{s}}b$ for any $a,b \in X^*$ and any $\m{s} \in \m{S}$, showing (2).

\par $(2) \Longrightarrow (1).$ Let $\m{S}\subseteq K\langle X\rangle$ be a set of monic polynomials, and let $f \in K\langle X\rangle$ be a polynomial. Using the same argument given in \cite[Lemma 2]{BokSurv}, we have
\[
 f = \sum\limits_{u \leqslant \overline{f}}\kappa u + \sum\limits_{a\overline{\m{s}}b \leqslant \overline{f}}\kappa'a\m{s}b,
\]
where $\kappa,\kappa' \in K$, and $u\in\mathrm{Irr}(\m{S})$. From this observation and $\mathrm{Irr}(\m{S})$ is a linear basis for $K\langle X \rangle /I(\m{S})$, we get that $(\m{s},\m{s}')_w \equiv 0 \bmod(\mathfrak{S},w)$ for any $\m{s},\m{s}'\in\m{S}$, and hence, $\m{S}$ is a Gr\"obner--Shirshov basis, finishing the proof.
\end{proof}

\smallskip

\begin{exas}
Let $K$ be an arbitrary commutative ring and consider the following algebra $\Lambda = K \langle x,y \rangle/(x^2 - y^2)$. Let us consider the polynomials $\varphi = x^2 - y^2$, $\psi = xy^2 - y^2x$, and let $y \leqslant x$. It is not hard to see that the set $\m{S} = \{\varphi,\psi\}$ is a Gr\"obner--Shirshov basis of $\Lambda$. Indeed,
\begin{align*}
  & (\varphi, \varphi)_{w}={\varphi}x - x {\varphi} = x^3-y^2x - (x^3 - xy^2) = \psi, & w =xxx\\
  & (\varphi,\psi)_{w} = \varphi y^2 - x\psi =x^2y^2 - y^2y^2 - (x^2y^2 - xy^2x) & w = x^2y^2,\\
  & \phantom{(\varphi,\psi)_{w}}= \psi x + y^2 \varphi. &
\end{align*}

Since the set $\m{S}$ is monic, then the set
\[
 \mathrm{Irr}(\m{S}) = \bigcup\limits_{n,m > 0}\Bigl\{1,x, xy, y^n,y^mx\Bigr\}
   \]
is the $K$-basis for $\Lambda$, by Theorem \ref{CDA}.\hfill$\square$
\end{exas}

\section{A basis of Leavitt path algebras with coefficients in commutative rings}
The main goal of this section is to give a basis of the Leavitt path algebra of a row-finite graph with coefficients in
a commutative ring, which is an analog of Zelmanov et. al.'s result \cite[Theorem 1]{aajz:lpaofgkd}.

\smallskip

We begin by recalling some general notions of graph theory: a (directed) graph
$E = \left(E^0, E^1, s, r\right)$ (or shortly $E = \left(E^0, E^1\right)$)
consists of two disjoint sets $E^0$ and $E^1$, called \emph{vertices} and \emph{edges}
respectively, together with two maps $s, r: E^1 \longrightarrow E^0$.  The
vertices $s(e)$ and $r(e)$ are referred to as the \emph{source} and the \emph{range}
of the edge~$e$, respectively. The graph is called \emph{row-finite} if
$|s^{-1}(v)|< \infty$ for all $v\in E^0$. All graphs in this paper will be assumed
to be row-finite. A graph $E$ is \emph{finite} if both sets $E^0$ and $E^1$ are finite
(or equivalently, when $E^0$ is finite, by the row-finite hypothesis).
A vertex~$v$ for which $s^{-1}(v)$ is empty is called a \emph{sink}; a vertex~$v$ for which
$r^{-1}(v)$ is empty is called a \emph{source}; and a vertex~$v$ is \emph{regular} iff
$0 < |s^{-1}(v)| < \infty$. A \emph{path} $p = e_{1} \dots e_{n}$ in a graph $E$ is a sequence of
edges $e_{1}, \dots, e_{n}$ such that $r(e_{i}) = s(e_{i+1})$ for $i
= 1, \dots, n-1$.  In this case, we say that the path~$p$ starts at
the vertex $s(p) := s(e_{1})$ and ends at the vertex $r(p) :=
r(e_{n})$, and has \emph{length} $|p| := n$. 
We consider the vertices in~$E^0$ to be paths of length~$0$.
If $p$ is a path of positive length in $E$, and if $v= s(p) = r(p)$, then~$p$ is a \emph{closed path
based at} $v$. A closed path based at~$v$, $p = e_{1} \dots e_{n}$, is
a \emph{closed simple path based at}~$v$ if $s(e_i) \neq v$ for
every $i > 1$. If $p = e_{1} \dots e_{n}$ is a closed path and all vertices
$s(e_{1}), \dots, s(e_{n})$ are distinct, then the subgraph
$(s(e_{1}), \dots, s(e_{n}); e_{1}, \dots, e_{n})$ of the graph $E$
is called a \emph{cycle}.  An edge~$f$ is an \emph{exit} for a path
$p = e_{1} \dots e_{n}$ if $s(f) = s(e_{i})$ but $f \ne e_{i}$ for
some $1 \le i \le n$. A graph $E$ is \emph{acyclic} if it has no cycles;
and the graph $E$ is said to be a \emph{no-exit graph} if no cycle in $E$ has an exit.
We say that a graph $E$ satisfies \emph{Condition} (L) if every cycle in $E$ has an exit.

\smallskip

\par Given a (row-finite) directed graph $E$ and a field $K$, Abrams and Aranda Pino in
\cite{ap:tlpaoag05}, and independently Ara, Moreno, and Pardo in \cite{amp:nktfga},
introduced the \emph{Leavitt path algebra} $L_K(E)$. These Leavitt path algebras generalize
the Leavitt algebras $L_K(1, n)$ of \cite{leav:tmtoar}, and also contain many other interesting
classes of algebras. In addition, Leavitt path algebras are intimately related to graph
$C^*$-algebras (see \cite{r:ga}). Later, in \cite{tomf:lpawciacr} Tomforde generalized the
construction of Leavitt path algebras by replacing the field with a commutative ring.

\begin{defn}[{\textit{cf.}~\cite[Definition~1.3]{ap:tlpaoag05} and \cite[Definition~2.4]{tomf:lpawciacr}}]
Let $K$ be a commutative ring and $E = (E^0, E^1,s,r)$ be a graph. The \emph{Leavitt path algebra} $L_K(E)$ of the graph~$E$ \emph{with coefficients in}~$K$ is the $K$-algebra presented by the
set of generators $E^0\cup E^1\cup (E^1)^{\ast}$ -- where $E^1\rightarrow (E^1)^{\ast}$,
$e\mapsto e^{\ast}$, is a bijection with $E^0$, $E^1$, $(E^1)^{\ast}$ pairwise disjoint -- satisfying the following relations:

(1) $v v' = \delta_{v,v'} v$ for all $v, v' \in E^0$;

(2) $s(e) e = e = e r(e)$, $r(e) e^{\ast} = e^{\ast} = e^{\ast} s(e)$ for all $e\in E^1$;

(3) $e^{\ast} f = \delta_{e,f} r(e)$ for all $e,f \in E^1$;

(4) $v = \sum\limits_{s(e)=v} e e^{\ast}$ whenever $v\in E^0$ is a regular vertex.
\end{defn}

Let us present a list of examples of some known algebras which can be described as a Leavitt path algebra. For more examples see \cite{Abrams-surv}.

\begin{exas}[{\bf Full matrix algebras}] Let $A_n$ denote the graph which is shown in fig.\ref{An}
\begin{figure}[h!]
  \[
   \xymatrix{
    \bullet^{v_1} \ar@{->}[r]^{e_1} & \bullet^{v_2} \ar@{->}[r]^{e_2} & \ldots & \bullet^{v_{n-1}} \ar@{->}[r]^{e_{n-1}} & \bullet^{v_n}
    }
  \]
  \caption{The graph $A_n$ is shown}\label{An}
\end{figure}
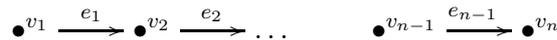

\smallskip

\par Then $L_K(A_n) \cong \mathbb{M}_n(K)$, the full $n\times n$ matrix algebra over $K$. Let us describe this isomorphism. Let $E_{i,j}$ be the standard matrix units, here $1 \le i,j\le n$, then the isomorphism $L_K(A_n) \cong \mathbb{M}_n(K)$ can be described as follows:
\begin{align*}
  &v_i \leftrightarrow E_{i,i}, & 1\le i\le n,\\
  &e_i \leftrightarrow E_{i,i+1}, &1 \le i \le n-1,\\
  &e_i^* \leftrightarrow E_{i+1,i}, & 1 \le i \le n-1.
\end{align*}\hfill$\square$
\end{exas}

\begin{exas}[{\bf The Laurent polynomial algebra}]\label{Loran}
Let $\Omega_1$ denote the graph
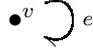
\[
   \xymatrix{
    \bullet^v \ar@(ur,dr)[]^e
    }
  \]

Then $L_K(\Omega) \cong K[t,t^{-1}]$, the Laurent polynomial algebra. The isomorphism is clear:
\begin{align*}
  &v \leftrightarrow 1,\\
  &e \leftrightarrow t,\\
  &e^* \leftrightarrow t^{-1}.
\end{align*}\hfill$\square$
\end{exas}

\begin{exas}[{\bf Leavitt algebra}]
For $1 \le \ell < \infty$, let $\Omega_\ell$ denote the graph with $\ell$ edges and with one vertex $v$ (see fig.\ref{O}).
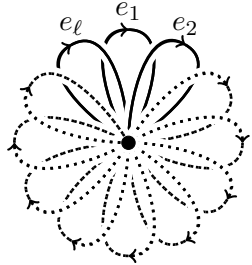
\begin{figure}[h!]
  \begin{center}
    \begin{tikzpicture}
       \begin{scope}
        \draw [line width = 5,white] (0,0) to [out = 120, in = 180] (0,1.5) to [out = 0, in =60] (0,0);
        \draw[line width =1] (0,0) to [out = 120, in = 180] (0,1.5) to [out = 0, in =60] (0,0);
        \draw[line width =1,->] (0,0) to [out = 120, in = 180] (0,1.5);
        \draw[line width =1] (0,1.5) to [out = 0, in =60] (0,0);
        \node [above] at (0,1.5) {$e_1$};
       \end{scope}
       \begin{scope}[rotate =30]
        \draw [line width = 5,white] (0,0) to [out = 120, in = 180] (0,1.5) to [out = 0, in =60] (0,0);
        \draw[line width =1] (0,0) to [out = 120, in = 180] (0,1.5) to [out = 0, in =60] (0,0);
        \draw[line width =1,->] (0,0) to [out = 120, in = 180] (0,1.5);
        \draw[line width =1] (0,1.5) to [out = 0, in =60] (0,0);
        \node [above] at (0,1.5) {$e_\ell$};
       \end{scope}
       \begin{scope}[rotate =330]
        \draw [line width = 5,white] (0,0) to [out = 120, in = 180] (0,1.5) to [out = 0, in =60] (0,0);
        \draw[line width =1] (0,0) to [out = 120, in = 180] (0,1.5) to [out = 0, in =60] (0,0);
        \draw[line width =1,->] (0,0) to [out = 120, in = 180] (0,1.5);
        \draw[line width =1] (0,1.5) to [out = 0, in =60] (0,0);
        \node [above] at (0,1.5) {$e_2$};
       \end{scope}
       \begin{scope}[rotate =300]
        \draw [line width = 5,white] (0,0) to [out = 120, in = 180] (0,1.5) to [out = 0, in =60] (0,0);
        \draw[line width =1,dotted] (0,0) to [out = 120, in = 180] (0,1.5) to [out = 0, in =60] (0,0);
        \draw[line width =1,->,dotted] (0,0) to [out = 120, in = 180] (0,1.5);
        \draw[line width =1,dotted] (0,1.5) to [out = 0, in =60] (0,0);
       \end{scope}
       \begin{scope}[rotate =270]
        \draw [line width = 5,white] (0,0) to [out = 120, in = 180] (0,1.5) to [out = 0, in =60] (0,0);
        \draw[line width =1,dotted] (0,0) to [out = 120, in = 180] (0,1.5) to [out = 0, in =60] (0,0);
        \draw[line width =1,->,dotted] (0,0) to [out = 120, in = 180] (0,1.5);
        \draw[line width =1,dotted] (0,1.5) to [out = 0, in =60] (0,0);
       \end{scope}
       \begin{scope}[rotate =240]
        \draw [line width = 5,white] (0,0) to [out = 120, in = 180] (0,1.5) to [out = 0, in =60] (0,0);
        \draw[line width =1,dotted] (0,0) to [out = 120, in = 180] (0,1.5) to [out = 0, in =60] (0,0);
        \draw[line width =1,->,dotted] (0,0) to [out = 120, in = 180] (0,1.5);
        \draw[line width =1,dotted] (0,1.5) to [out = 0, in =60] (0,0);
       \end{scope}
       \begin{scope}[rotate =210]
        \draw [line width = 5,white] (0,0) to [out = 120, in = 180] (0,1.5) to [out = 0, in =60] (0,0);
        \draw[line width =1,dotted] (0,0) to [out = 120, in = 180] (0,1.5) to [out = 0, in =60] (0,0);
        \draw[line width =1,->,dotted] (0,0) to [out = 120, in = 180] (0,1.5);
        \draw[line width =1,dotted] (0,1.5) to [out = 0, in =60] (0,0);
       \end{scope}
       \begin{scope}[rotate =180]
        \draw [line width = 5,white] (0,0) to [out = 120, in = 180] (0,1.5) to [out = 0, in =60] (0,0);
        \draw[line width =1,dotted] (0,0) to [out = 120, in = 180] (0,1.5) to [out = 0, in =60] (0,0);
        \draw[line width =1,->,dotted] (0,0) to [out = 120, in = 180] (0,1.5);
        \draw[line width =1,dotted] (0,1.5) to [out = 0, in =60] (0,0);
       \end{scope}
       \begin{scope}[rotate =150]
        \draw [line width = 5,white] (0,0) to [out = 120, in = 180] (0,1.5) to [out = 0, in =60] (0,0);
        \draw[line width =1,dotted] (0,0) to [out = 120, in = 180] (0,1.5) to [out = 0, in =60] (0,0);
        \draw[line width =1,->,dotted] (0,0) to [out = 120, in = 180] (0,1.5);
        \draw[line width =1,dotted] (0,1.5) to [out = 0, in =60] (0,0);
       \end{scope}
       \begin{scope}[rotate =120]
        \draw [line width = 5,white] (0,0) to [out = 120, in = 180] (0,1.5) to [out = 0, in =60] (0,0);
        \draw[line width =1,dotted] (0,0) to [out = 120, in = 180] (0,1.5) to [out = 0, in =60] (0,0);
        \draw[line width =1,->,dotted] (0,0) to [out = 120, in = 180] (0,1.5);
        \draw[line width =1,dotted] (0,1.5) to [out = 0, in =60] (0,0);
       \end{scope}
       \begin{scope}[rotate =90]
        \draw [line width = 5,white] (0,0) to [out = 120, in = 180] (0,1.5) to [out = 0, in =60] (0,0);
        \draw[line width =1,dotted] (0,0) to [out = 120, in = 180] (0,1.5) to [out = 0, in =60] (0,0);
        \draw[line width =1,->,dotted] (0,0) to [out = 120, in = 180] (0,1.5);
        \draw[line width =1,dotted] (0,1.5) to [out = 0, in =60] (0,0);
       \end{scope}
       \begin{scope}[rotate =60]
        \draw [line width = 5,white] (0,0) to [out = 120, in = 180] (0,1.5) to [out = 0, in =60] (0,0);
        \draw[line width =1,dotted] (0,0) to [out = 120, in = 180] (0,1.5) to [out = 0, in =60] (0,0);
        \draw[line width =1,->,dotted] (0,0) to [out = 120, in = 180] (0,1.5);
        \draw[line width =1,dotted] (0,1.5) to [out = 0, in =60] (0,0);
       \end{scope}
       \draw [fill](0,0) circle (2.5pt);
    \end{tikzpicture}
  \end{center}
  \caption{The graph  $\Omega_\ell$ is shown.}\label{O}
\end{figure}

Then $L_K(\Omega_\ell) \cong L_K(1,\ell)$, the Leavitt algebra of order $\ell$, which is an $K$-algebra, defined by the generators $\{x_i,y_i:1 \le i \le \ell\}$, and relations
\[
y_ix_i = \delta_{i,j}, \qquad \sum\limits_{i=1}^\ell x_iy_i =1.
\]\hfill$\square$
\end{exas}

\begin{exas}[{\bf The Toeplitz algebra}]\label{Toeplitz}
For any field $K$, the Jacobson algebra, described in \cite{J}, is the $K$-algebra:
\[
A = K \langle x,y: xy =1 \rangle.
\]
This algebra was the first example appearing in the literature of an algebra which is not directly finite, that is, in which there are elements $x,y$ for which $xy =1$ but $yx \ne 1$. Let $\mathcal{T}$ denote the ``Toeplitz graph'', which is shown on the fig.\ref{T}
\begin{figure}[h!]
  \[
  \xymatrix{
    \bullet^v \ar@(ul,dl)_e \ar@{->}[r]^f& \bullet^u
   }
   \]
  \caption{The Toeplitz graph $\mathcal{T}$ is shown.}\label{T}
\end{figure}
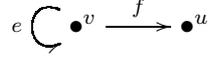

Then $L_K(\mathcal{T}) \cong A$. The isomorphism is described as follows:
\begin{align*}
  &v \leftrightarrow yx,\\
  &u \leftrightarrow 1-yx,\\
  &e \leftrightarrow y^2x,\\
  &f \leftrightarrow y-y^2x,\\
  &e^* \leftrightarrow yx^2,\\
  &f^*\leftrightarrow x - yx^2.
\end{align*}\hfill$\square$
\end{exas}

Notice that if the graph $E$ is finite, then $L_K(E)$ is a unital ring having identity
$1 = \sum\limits_{v\in E^0}v$ (see, e.g., \cite[Lemma 1.6]{ap:tlpaoag05} and \cite[Subsection 4.2]{tomf:lpawciacr}). It is easy to see that the mappings given by $v \mapsto v$, for $v \in E^0$, and $e \longmapsto e^{\ast}$, $e^{\ast} \longmapsto e$ for $e\in E^1$, produce an involution on the algebra $L_K(E)$, and for any path $p = e_{1} \cdots e_{n} $ there exists $p^{\ast} := e_{n}^{\ast} \cdots e_{1}^{\ast}$. Also, $L_K(E)$ has the following \textit{universal} property: if $A$ is an $K$-algebra generated by a family of elements $\left\{a_{v}, b_{e}, c_{e^{\ast}} \mid v\in E^0, e\in E^1, {e^{\ast } \in (E^1)^{\ast}}\right\}$ satisfying the analogous to (1) -- (4) relations in Definition~2.1, then there always exists a unique $K$-algebra homomorphism
$\varphi: L_K(E) \rightarrow A$ given by ${\varphi(v) = a_{v}}$,
${\varphi(e) = b_{e}}$ and ${\varphi(e^{\ast}) = c_{e^{\ast}}}$.

\smallskip

\par In \cite[Theorem 1]{aajz:lpaofgkd} Zelmanov et. al. found a basis of the Leavitt path algebra of a finite graph with coefficients in a field. Here we note that the result extends to the case of Leavitt path algebras with coefficients in a commutative ring.

Namely, for an arbitrary vertex $v\in E^0$ which is not a sink, choose an edge $e\in E^1$ such that $s(e) = v$. We will refer to this edge as \emph{special}. In other words, we fix a function $\gamma: E^0\setminus \{sinks\} \longrightarrow E^1$ such that $s(e) = v$ for an arbitrary vertex $v\in E^0\setminus \{sinks\}$.

\smallskip

\par As in \cite[Theorem 1]{aajz:lpaofgkd} we consider the the order $<$ on the set of generators $X:=E^0 \cup E^1 \cup \left(E^1\right)^*$. Chose an arbitrary well-ordering on the set of vertices $E^0$. If $e$, $f$ are edges and $s(e) < s(f)$, then $e < f$. It remains to order the edges that have the same source. Let $v$ be a vertex which is not a sink. Let $e_1, \ldots, e_\ell$ be all the edges that originate from $v$. Suppose $e_1$ is a special edge. We order the edges as follows: $e_1 > e_2 >\ldots > e_\ell$. Choose an arbitrary well-ordering on the set $E^1$. For arbitrary elements $v\in E^0$, $e \in E^1$, $f^* \in \left(E^1\right)^*$, we let $v<e<f^*$. Thus the set $E^0 \cup E^1 \cup \left(E^1\right)^*$ is well-ordered. Let $X^*$ be the set of all words in the alphabet $X$. The length-lex order makes $X^*$ a well-ordered set. With respect to this order, we have the follow useful fact.

\begin{lem}\label{GSB-LPA}
Let $E$ be a row-finite graph and $K$ a commutative ring. Then, the Gr\"obner--Shirshov basis $\m{S}$ of the Leavitt path algebra $L_K(E)$ can be described as follows
\[
\m{S} = \m{S}_1 \cup \m{S}_2 \cup \m{S}_3\cup\m{S}_4,
\]
where
\begin{align*}
  &\m{S}_1 = \bigcup\limits_{v,u\in E^0,\,e\in E^1}\left\{vu - \delta_{v,u}v,\, ve - \delta_{v, s(e)}e,\,ev - \delta_{v,r(e)}e \right\},\\
  &\m{S}_2 = \bigcup\limits_{v\in E^0,\,e,f \in E^1}\left\{ve^* - \delta_{v, r(e)}e^*,\,e^*v - \delta_{v,s(e)}e^*,\, e^*f - \delta_{e,f}r(e) \right\},\\
  &\m{S}_3 = \bigcup\limits_{e,f\in E^1}\left\{ee^* -s(e) + \sum\limits_{s(f)=s(e)} ff^*: \gamma(s(e))\ne f\right\},\\
  &\m{S}_4 = \bigcup\limits_{e,f\in E^1}\Bigl(\left\{ef:r(e) \ne s(f)\right\}\cup \left\{e^*f^*: r(f) \ne s(e)\right\}\cup \left\{ef^*: r(e) \ne r(f)\right\}\Bigr).
\end{align*}
\end{lem}
\begin{proof}
Let us consider the polynomials
  \[
  \varphi = e^*g - \delta_{e,g}r(e), \qquad \rho = ee^*-s(e) + \sum_{\substack{s(f) =s(e)\\ \gamma(s(e)) \ne f}}ff^*.
  \]
We have
\begin{eqnarray*}
  (\varphi,\rho)_{ee^*g}&=& \left(ee^* + \sum_{\substack{s(f) =s(e)\\ \gamma(s(e)) \ne f}}ff^* - s(e)\right)g - e \Bigl(e^*g - \delta_{e,g}r(e)\Bigr)=\\
  &=&\delta_{e,g}er(e) + \delta_{f,g}gr(g) - \delta_{s(e),s(g)}g - \delta_{e,g}er(e) + \delta_{e,g}er(e)=\\
  &=& \delta_{e,g}e + \delta_{f,g}g - \delta_{s(e),s(g)}g,
\end{eqnarray*}
it follows that
\[
 (\varphi,\rho)_{ee^*g} \equiv 0 \left(\m{S}, ee^*g\right).
\]
The straightforward computations show that the set of another polynomials is closed with respect to compositions.
\end{proof}

\begin{thm}[{\textit{cf.}~\cite[Theorem 1]{aajz:lpaofgkd}}]\label{B-LPA}
For any graph $E= \left(E^0, E^1, s, r\right)$ and any commutative ring $K$, the following elements form a basis of the Leavitt path algebra $L_K(E):$
\begin{itemize}
\item[(1)] the set of all vertices $E^0$;
\item[(2)] $p, p^{\ast},$ where $p$ is a path of positive length in $E$;
\item[(3)] $pq^{\ast},$ where $p = e_1\cdots e_n$, $q= f_1\cdots f_m$, $e_i, f_j\in E^1$, are paths of positive length that end at the same vertex $r(e_n) = r(f_m)$, with the condition that the last edges $e_n$ and $f_m$ are either distinct or equal, but not special.
\end{itemize}
\end{thm}

\begin{proof}
By Theorem \ref{CDA} and Lemma \ref{GSB-LPA}, the set of irreducible words (not containing
the leading words of relations in Lemma \ref{GSB-LPA} as subwords) is an $K$-basis of
$L_K(E)$, which finishes the proof.
\end{proof}

As an application of Theorem \ref{B-LPA}, we get the following useful facts.

\begin{cor}\label{cor3.4}
Let $E = \left(E^0, E^1\right)$ be a finite graph, $K$ a commutative ring and $K_0 := \prod\limits_{v\in E^0}K.$ Let $\left\{a_v\ |\ v\in E^0\right\}$ be the standard basis of the free $K$-module $K_0$. Then the following statements hold:
\begin{itemize}
\item[(1)] the ring homomorphism $\varphi: K_0\longrightarrow L_K(E)$, defined by $\varphi(a_v) = v$ for all $v\in E^0$, is injective;
\item[(2)] $\mathrm{Im}(\varphi)$ is a direct summand of $L_K(E)$ as an $K_0$-module;
\item[(3)] $L_K(E)$ is a projective $K_0$-module.
\end{itemize}
\end{cor}
\begin{proof}
The statements (1) and (2) follow immediately from Theorem \ref{B-LPA}.

(3) For any $v\in E^0$ we denote by $\pi_v$ the $v$-th canonical surjection from $K_0$ onto $K$. We then have that
\[
L_K(E) = \bigoplus_{v\in E^0}\mathrm{Im}(\pi_v)L_K(E)
\]
as $K_0$-modules. For any $v\in E^0$, by Theorem \ref{B-LPA}, $\mathrm{Im}(\pi_v)L_K(E)$ is a free $K$-module with a basis consisting the following elements:
\begin{itemize}
\item[(1)] $p$, where $p$ is a path in $E$ starting at $v$;
\item[(2)] $q^*$, where $q$ is a path of positive length in $E$ ending at $v$;
\item[(3)] $pq^{\ast},$ where $p = e_1\cdots e_n$, $q= f_1\cdots f_m$, $e_i, f_j\in E^1$, are paths of positive length that end at the same vertex $r(e_n) = r(f_m)$ and $s(p) = v$, with the condition that the last edges $e_n$ and $f_m$ are either distinct or equal, but not special.
\end{itemize}

Furthermore, $K$ is clearly a projective $K_0$-module by pullback along $\pi_v$, and hence, $\mathrm{Im}(\pi_v)L_K(E)$ is also a projective $K_0$-module. From these observations, we immediately get that $L_K(E)$ is a projective $K_0$-module, finishing the proof.
\end{proof}

The structure of the Leavitt path algebra of a finite no-exit graph with coefficients in a field is given in \cite[Theorems 3.8 and 3.10]{apm:lflpa} (see, also, \cite[Theorem 2.1]{al:nlpaatra}) by Abrams, Aranda Pino and Siles Molina. Using Theorem \ref{B-LPA} and repeating verbatim the proofs of \cite[Theorems 3.8 and 3.10]{apm:lflpa}, we note that the result extends to the case of Leavitt path algebras with coefficients in a commutative ring.

\begin{cor}\label{cor3.5} Let $K$ be a commutative ring, $E$ a finite no-exit graph,
$\{c_1,\ldots, c_l\}$ the set of cycles, and $\{v_1, \ldots, v_k\}$ the set of sinks.
Then $L_K(E)$ is isomorphic to
\[
\bigoplus_{i=1}^l \bigoplus_{j=1}^{k}\mathbb{M}_{m_i}\left(K[x,x^{-1}]\right)\oplus \mathbb{M}_{n_j}(K)
\]
where for each $1\leq i\leq l$, $m_i$ is the number of paths ending in a fixed (although arbitrary) vertex of the cycle $c_i$ which do not contain the cycle itself, and for each $1\leq j\leq k$, $n_j$ is the number of paths ending in the sink $v_j$.
\end{cor}

\section{The homological dimensions of Leavitt path algebras}
The main goal of this section is to give sharp bounds for the homological dimensions of the Leavitt path algebra $L_K(E)$ of a finite graph $E$ with coefficients in a commutative ring $K$, as well as establish a formula for the homological dimensions of $L_K(E)$ when $K$ is a commutative unital algebra over a field.

We begin by recalling some general notions of the theory of homological dimensions. We shall say that the left $R$-module $M$ has \emph{projective dimension} $\leq n$ if
$M$ has a projective resolution $0\rightarrow P_{n}\rightarrow\cdots\rightarrow P_0\rightarrow M\rightarrow 0$, where each $P_i\ (i= 0,\ldots, n)$ is a projective left $R$-module. The least such integer $n$ is called the projective dimension of $M$ and denoted by $\mathrm{pd_R}(M)$. If no such integer exists the dimension is denoted to be $\infty$. The \emph{left global dimension} of a ring $R$ is defined to be
\[
\mathrm{l.gl.dim}(R) := \sup\{\mathrm{pd}_R(M): M\in |R\mbox{-}\mathrm{Mod}|\}\leq\infty.
\]

The \emph{right global dimension} of $R$, denoted by $\mathrm{r.gl.dim}(R)$, is defined similarly. If $R$ is a commutative ring, we shall write $\mathrm{gl.dim}(R)$ for the common value of $\mathrm{l.gl.dim}(R)$ and $\mathrm{r.gl.dim}(R)$.

We next go on define the flat dimensions of left $R$-modules. We shall say that the left $R$-module $M$ has \emph{flat dimension} $\leq n$ if $M$ has a resolution $0\rightarrow F_{n}\rightarrow\cdots\rightarrow F_0\rightarrow M\rightarrow 0$, where each $F_i\ (i= 0,\ldots, n)$ is a flat left $R$-module. The least such integer $n$ is called the flat dimension of $M$ and denoted by $\mathrm{fd_R}(M)$.

If no such integer exists the dimension is denoted to be $\infty$. The \emph{left weak global
dimension} of a ring $R$ is defined to be
\[
\mathrm{w.gl.dim}(R) = \sup\{\mathrm{fd_R}(M): M\in |R\mbox{-}\mathrm{Mod}|\}\leq\infty.
\]

The \emph{right weak global dimension} of $R$ is defined similarly, by using right $R$-modules. It is well-known that the left and right weak global dimensions of $R$ are the same (see, e.g., \cite[Theorem 5.63, page 185]{l:lomar}). The weak global dimension of the ring $R$ is at most equal to its left or right global dimension. For another properties of the homological dimensions we may refer to \cite{ce:ha} and \cite{l:lomar}, for example.

The following lemma provides us with a sharp lower bound for the homological dimensions of
Leavitt path algebras.

\begin{lem}\label{lem4.1}
Let $K$ be a commutative ring and $E$ a finite graph. Then we have
\[
\mathrm{gl.dim}(K) \leq \mathrm{l.gl.dim}\left(L_K(E)\right) \text{ and } \mathrm{w.gl.dim}(K) \leq \mathrm{w.gl.dim}\left(L_K(E)\right).
\]
\end{lem}
\begin{proof}
We start to an $K$-algebra $K_0 = \prod\limits_{v\in E^0}K.$ By Corollary \ref{cor3.4} (1) and (2), the ring $K_0$ may be considered as be to a subring of $L_K(E)$ such that $K_0$  is a direct summand of $L_K(E)$ as an $K_0$-module. Therefore, we may write $L_K(E) = K_0\oplus M$ for some $K_0$-module $M$.

Take any $K_0$-module $A$, and set $B = \mathrm{Hom}_{K_0}\left(L_K(E), A\right)$. Then $B$ carries the structure of a left $L_K(E)$-module, and $B \cong A\oplus \mathrm{Hom}_{K_0}(M, A)$ (as
$K_0$-modules). This implies that
\[
\mathrm{pd_{K_0}}(A) \leq \mathrm{pd_{K_0}}(B) \text{ and }
\mathrm{fd_{K_0}}(A) \leq \mathrm{fd_{K_0}}(B),
\]
by \cite[Exercises 7, page 123]{ce:ha}. Since Corollary \ref{cor3.4} (3), $L_K(E)$ is a projective $K_0$-module, so
\[
\mathrm{pd_{K_0}}(B) \leq \mathrm{pd_{L_K(E)}}(B) \text{ and } \mathrm{fd_{K_0}}(B) \leq \mathrm{fd_{L_K(E)}}(B),
\]
by \cite[Exercises 10, page 123]{ce:ha}. But then
\[
\mathrm{pd_{K_0}}(A) \leq \mathrm{pd_{L_K(E)}}(B) \text{ and } \mathrm{fd_{K_0}}(A) \leq \mathrm{fd_{L_K(E)}}(B).
\]

Passing to the homological dimensions, we get that
\[
\mathrm{gl.dim}(K) = \mathrm{gl.dim}(K_0) \leq \mathrm{l.gl.dim}\left(L_K(E)\right)
\]
and
\[
\mathrm{w.gl.dim}(K) = \mathrm{w.gl.dim}(K_0) \leq \mathrm{w.gl.dim}\left(L_K(E)\right),
\]
finishing the proof.
\end{proof}

Our next goal is to give a sharp upper bound for the homological dimensions of Leavitt path algebras. To do so, we first need the following important lemma.

\begin{lem}\label{lem4.2}
Let $K$ be a commutative ring and $E = \left(E^0, E^1, r, s\right)$ a finite graph with that $\{v_1, v_2,\ldots, v_k\}\subseteq E^0$ is the set of regular vertices. For each $1\leq i\leq k$,
let $E_i$ denote the graph with the same vertices as $E$ and $v_i$ emits the same edges as it does in $E$, but all other vertices do not emit any edge. Then the following statements hold:
\begin{itemize}
\item[(1)] $\mathrm{l.gl.dim}\left(L_K(E)\right)\leq \max\limits_{1 \le i \le k}\Bigr\{\mathrm{gl.dim}(K) +1,\ \mathrm{l.gl.dim}\bigl(L_K(E_i)\bigr)\Bigr\}$,
\item[(2)] $\mathrm{w.gl.dim}\left(L_K(E)\right)\leq \max\limits_{1 \le i \le k}\Bigl\{\mathrm{w.gl.dim}(K) +1,\ \mathrm{w.gl.dim}\bigl(L_K(E_i)\bigr)\Bigr\}.$
\end{itemize}
\end{lem}
\begin{proof}
Let $K_0 = \prod\limits_{v\in E^0}K$ and we denote by $\left\{a_v\ |\ v\in E^0\right\}$ the standard basis of the free $K$-module $K_0$. Let $\varphi_i: K_0\longrightarrow L_K(E_i)$ and
$\psi_i: L_K(E_i)\longrightarrow L_K(E)$ $(1\leq i\leq k)$ be the ring homomorphisms which are defined respectively by: $\varphi_i(a_v) = v$,  and $\psi_i(v) = v$, $\psi_i(e) = e$ and
$\psi_i(e^{\ast}) = e^{\ast}$.

We claim that the following diagram

$$
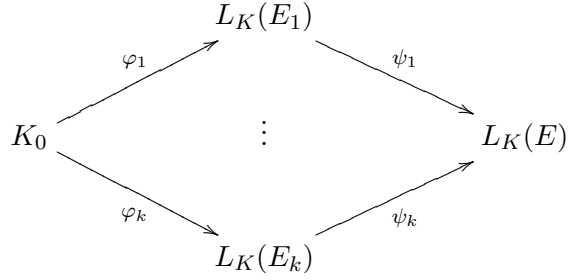
\xymatrix{
 && L_K(E_1)\ar[drr]^{\psi_{1}}\\
 K_0 \ar[urr]^{\varphi_{1}} \ar[drr]_{\varphi_{k}} && \vdots && L_K(E)\\
 && L_K(E_k)\ar[urr]_{\psi_{k}}
 }
$$
is a pushout in the category of rings. Indeed, it is easy to see that the above diagram is commutative, that means, $\psi_i\varphi_i = \psi_j\varphi_j$ for all $1 \le i,j \le k$.

Assume that $S$ is a ring and $\alpha_i: L_K(E_i)\longrightarrow S$ $(1\leq i\leq k)$ are ring
homomorphisms such that the following diagram
$$
 
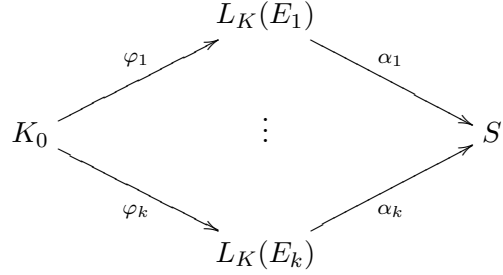
\xymatrix{
 && L_K(E_1)\ar[drr]^{\alpha_{1}}\\
 K_0 \ar[urr]^{\varphi_{1}} \ar[drr]_{\varphi_{k}} && \vdots && S\\
 && L_K(E_k)\ar[urr]_{\alpha_{k}}
 }
$$
is commutative, that means, $\alpha_i\varphi_i = \alpha_j\varphi_j$ for all $1 \le i,j \le k$. We then have that $\alpha_i(v) = \alpha_j(v)$ $(1\leq i, j\leq k)$, $\sum\limits_{v\in E^0_i}\alpha_i(v) = 1$ and $\alpha_i(v)\alpha_i(v') = \delta_{v, v'}\alpha_i(v)$.

For any $v\in E^0$, let $x_v = \alpha_1(v)$ and we have a family $\{x_v\ |\ v\in E^0\}$ of
orthogonal idempotents in $S$ such that $\sum\limits_{v\in E^0}x_v = 1$.

For any $1\leq i\leq k$ and any $e\in s^{-1}(v_i)$, let $y_e = \alpha_i(e)$ and
$z_{e^*} = \alpha_{i}(e^*)$. We then have that
$$
x_{s(e)}y_e = \alpha_1(s(e))\alpha_i(e)
= \alpha_i(s(e))\alpha_i(e)= \alpha_i(s(e)e) =\alpha_i(e)=y_e,
$$
$$
y_{e}x_{r(e)} = \alpha_i(e)\alpha_1(r(e)) = \alpha_i(e)\alpha_i(r(e))= \alpha_i(er(e)) =
\alpha_i(e)=y_e.
$$

Similarly $$x_{r(e)}z_{e^*} = z_{e^*}\ \text{ and }\ z_{e^*}x_{s(e)} = z_{e^*}.$$

Take any $e$ and $f\in E^1$. If there exists a regular vertex $v_i$ such that $s(e) = v_i=s(f)$, then
$$
z_{e^*}y_f = \alpha_i(e^*)\alpha_i(f) =  \alpha_i(e^*f)=  \alpha_i(\delta_{e, f}r(e))
= \delta_{e, f}\alpha_i(r(e))=\delta_{e, f}x_{r(e)}.
$$

Otherwise, we get that $s(f) \neq s(e)$. But then
$$
z_{e^*}y_f = \left(z_{e^*}x_{s(e)}\right) \left(x_{s(f)}y_f\right)= z_{e^*}\left(x_{s(e)}x_{s(f)}\right)y_f=0.
$$

For any regular vertex $v_i$, we have that
$$
\sum_{e\in s^{-1}(v_i)}y_ez_{e^*}= \sum_{e\in s^{-1}(v_i)}\alpha_i(e)\alpha_i(e^*) =
\alpha_i\sum_{e\in s^{-1}(v_i)}ee^*= \alpha_i(v_i)=x_{v_i}.
$$

From these observations, we have a family
$$
\left\{x_e, y_e, z_{e^*}\mid v\in E^0, e\in E^1,
e^{\ast}\in (E^1)^{\ast}\right\}
$$
of elements in $S$ satisfying the relations (1) -- (4) of Definition 3.1. By the universal property of $L_K(E)$, there exists a unique ring homomorphism $\alpha: L_K(E)\longrightarrow S$ such that $\alpha(v) = x_v$, $\alpha(e)= y_e$ and $\alpha(e^*) = z_{e^*}$. Also, it is easy to check that $\alpha_i = \alpha\psi_i$ $(1\leq i\leq k)$, thereby establishing the claim.

Using this observation, Corollary \ref{cor3.4} (3) and \cite[Theorem 9 (iii)]{dicks:mvpocor}, we
immediately have that $L_K(E)$ is a flat left $L_K(E_i)$-module, so we get the desired
conclusions, by \cite[Corollary 7, and Remarks (i) and (ii), pape 567]{dicks:mvpocor}.
\end{proof}

For clarification, we illustrate the ideas which arise in Lemma \ref{lem4.2} by presenting
the following example.

\begin{exas}
Let $E$ be the graph

$$
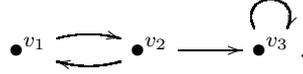
\xymatrix{\bullet^{v_1}\ar@/^.5pc/[r]  &\bullet^{v_2} \ar@/^.5pc/[l]\ar[r]& \bullet^{v_3}\ar@(ul,ur)}.$$
Clearly, $v_i$'s are regular vertices, and hence, we have the graphs $E_i$ ($1\leq i\leq 3$) as follows:

$$E_1 = \xymatrix{\bullet^{v_1}\ar[r]  &\bullet^{v_2}& \bullet^{v_3}}$$

$$E_2 = \xymatrix{\bullet^{v_1}&\bullet^{v_2}\ar[l]  \ar[r] & \bullet^{v_3}}$$

$$E_3 = \xymatrix{\bullet^{v_1}&\bullet^{v_2}& \bullet^{v_3}\ar@(ul,ur)}.$$

For any commutative ring $K$, by Corollary \ref{cor3.5}, we obtain that
$$
L_K(E_1) \cong \mathbb{M}_2(K)\oplus K,\ L_K(E_2) \cong \mathbb{M}_2(K)\oplus \mathbb{M}_2(K)
$$
and
$$
L_K(E_3) \cong K\oplus K\oplus K[x, x^{-1}].
$$

But then, as the homological dimensions are Morita invariants, the first two displayed ring isomorphisms yield that
\[
\mathrm{l.gl.dim}\left(L_K(E_1)\right) = \mathrm{gl.dim}(K) = \mathrm{l.gl.dim}\left(L_K(E_2)\right)
\]
and
\[
\mathrm{w.gl.dim}\left(L_K(E_1)\right) = \mathrm{w.gl.dim}(K) = \mathrm{w.gl.dim}\left(L_K(E_2)\right).
\]
Also, using \cite[Theorem 1.7]{douglas:twgdogroar} and \cite[Theorem 2]{dicks:mvpocor},
we immediately get that $\mathrm{gl.dim}\left(K[x, x^{-1}]\right) = \mathrm{gl.dim}(K)+1$ and $\mathrm{w.gl.dim}\left(K[x, x^{-1}]\right) = \mathrm{w.gl.dim}(K)+1$, respectively. From these observations and the last displayed ring isomorphism yield that
\[
\mathrm{l.gl.dim}\left(L_K(E_3)\right) = \mathrm{gl.dim}(K) +1,\ \mathrm{w.gl.dim}\left(L_K(E_3)\right) = \mathrm{w.gl.dim}(K)+1.
\]
From these observations, and Lemmas \ref{lem4.1} and \ref{lem4.1}, we get that
\[
\mathrm{gl.dim}(K)\leq \mathrm{l.gl.dim}\left(L_K(E)\right) \leq \mathrm{gl.dim}(K) +1
\]
and
\[
\mathrm{w.gl.dim}(K)\leq \mathrm{w.gl.dim}\left(L_K(E)\right) \leq \mathrm{w.gl.dim}(K) +1.
\]
\hfill$\square$
\end{exas}

Lemma \ref{lem4.2} allows us to reduce calculating the homological dimensions of the Leavitt path algebra $L_K(E)$ to calculating the homological dimensions of $L_K(F)$, where $F$ is the graph with the same vertices as $E$ and its all edges are only edges emitting from some regular vertex $v\in E^0$ (other vertices do not emit any edge). Notice that $F$ is a no-exit graph if and only if $E$ contains a loop based at $v$, which has no an exit. In this case, we may calculate easily the homological dimensions of $L_K(F)$, by using Corollary \ref{cor3.5}, \cite[Theorem 1.7]{douglas:twgdogroar} and \cite[Theorem 2]{dicks:mvpocor}. Therefore, the difficulty is how to calculate the homological dimensions of $L_K(F)$ when $F$ contains a loop having an exit. To remedy this problem, our idea is based on Abrams, Louly, Pardo and Smith's useful result \cite[Proposition 1.8]{alps:fiitcolpa}. But, we first recall the following notion.

\begin{defn}[{\textit{cf.}~\cite[Definition~1.6]{alps:fiitcolpa}}]
Let $E = \left(E^0, E^1, r, s\right)$ be a graph, and let $v\in E^0$. Let $v^*$ and $f$ be symbols
not in $E^0\cup E^1$. We form the \emph{expansion graph} $E_v$ from $E$ at $v$ as follows:
\[E^0_v = E^0 \cup \{v^*\},\] \[E^1_v = E^1 \cup \{f\}, \ \]

\begin{equation*}
s_{E_v}(e) = \left\{
\begin{array}{lcl}
v&  & \text{if }e = f, \\
&  &  \\
v^*&  & \text{if }s_{E}(e) = v,\\
&  &  \\
s_{E}(e)&  & \text{otherwise},%
\end{array}%
\right.
\end{equation*}%

\begin{equation*}
r_{E_v}(e) = \left\{
\begin{array}{lcl}
v^*&  & \text{if }e = f, \\
&  &  \\
r_{E}(e)&  & \text{otherwise. \ \ }%
\end{array}%
\right.
\end{equation*}%
\end{defn}

\begin{exas}
Let $E$ be the graph
$$
\xymatrix{\bullet^{v}\ar@(dl,ul)\ar[r]& \bullet^{w}\ar@(ul,ur)}.
$$ Then the expansion graph $E_v$ is
$$
\xymatrix{\bullet^{v}\ar@/^.5pc/[r]^{f}  &\bullet^{v^*} \ar@/^.5pc/[l]\ar[r]& \bullet^{w}\ar@(ul,ur)}.$$\hfill$\square$
\end{exas}

In \cite[Proposition 1.8]{alps:fiitcolpa}, Abrams, Louly, Pardo and Smith showed that if $K$ is a field and $E$ is a finite graph such that $L_K(E)$ is simple, then $L_K(E)$ is Morita equivalent to $L_K(E_v)$. In the following proposition, we extend this result to the Leavitt path algebra of a finite graph satisfying Condition (L) with coefficients in a commutative ring. Recall a graph $E$ is said to satisfy \emph{Condition} (L) if every cycle in $E$ has an exit.

\begin{prop}\label{prop4.6}
Let $K$ be a commutative ring, $E$ a finite graph satisfying Condition \textsc{(L)}, and let $v\in E^0$. Then $L_K(E)$ is Morita equivalent to $L_K(E_v)$.
\end{prop}
\begin{proof} We essentially follow the ideas in the proof of \cite[Proposition 1.8]{alps:fiitcolpa}.

For each $w\in E^0$, define $Q_w = w$. For each $e\in s^{-1}(v)$, define $T_e = fe$ and $T_{e^*} = f^*e^*$. For $e\in E^1$ otherwise, define $T_e = e$ and $T_{e^*} = e^*$. It is easy to check that $\{Q_w, T_e, T_{e^*}\mid w\in E^0, e\in E^1, e^*\in (E^1)^*\}$ is a family of elements in $L_R(E_v)$ satisfying the relations (1) - (4) of Definition 3.1. Note that
$$
\sum_{e\in s^{-1}(v)}T_eT_{e^*} = f\sum_{e\in s^{-1}(v)}ee^*f^* = fvf^*
= ff^* = r(f) = v= Q_v.
$$

Therefore, by the universal property of $L_K(E)$, there is an $K$-algebra homomorphism $\pi: L_K(E)\longrightarrow L_K(E_v)$ such that $\pi(w) = Q_w$, $\pi(e)= T_e$, and $\pi(e^*) = T_{e^*}$. Repeating verbatim the proof of \cite[Proposition 1.8]{alps:fiitcolpa}, we get that
$$
\pi(L_K(E)) = \pi\left(1_{L_K(E)}\right)L_K(E_v)\pi\left(1_{L_K(E)}\right),
$$
where $\pi\left(1_{L_K(E)}\right)= \sum\limits_{w\in E^0}w$. Furthermore, we always have
$$
v^* = f^*f = f^*\sum_{w\in E^0}wf = f^*\pi\left(1_{L_K(E)}\right)f
\in L_K(E_v)\pi\left(1_{L_K(E)}\right)L_K(E_v),
$$
so $L_K(E_v)\pi\left(1_{L_K(E)}\right)L_K(E_v)=L_K(E_v)$. This implies that $L_K(E_v)$ is Morita equivalent to $\pi\left(1_{L_K(E)}\right)L_K(E_v)\pi\left(1_{L_K(E)}\right)$.

On the other hand, by \cite[Proposition 3.4]{tomf:lpawciacr}, $\pi(rw) = rQ_w = rw \neq 0$ for all $w\in E^0$ and $0\neq r\in K$. From this observation and $E$ is a graph satisfying Condition (L), we get that the homomorphism $\pi$ is injective, by \cite[Theorem 6.5]{tomf:lpawciacr}. Therefore, $L_K(E_v)$ is Morita equivalent to $L_K(E)$, finishing our proof.
\end{proof}

Using Proposition \ref{prop4.6}, Lemmas \ref{lem4.1} and \ref{lem4.2}, we establish sharp bounds for the homological
dimensions of Leavitt path algebras.

\begin{thm}\label{thm4.7}
Let $K$ be a commutative ring and $E$ a finite graph. Then the following statements hold:
\begin{itemize}
\item[(1)] $\mathrm{gl.dim}(K)\leq \mathrm{l.gl.dim}\left(L_K(E)\right) \leq \mathrm{gl.dim}(K) +1$;
\item[(2)] $\mathrm{w.gl.dim}(K)\leq \mathrm{w.gl.dim}\left(L_K(E)\right) \leq \mathrm{w.gl.dim}(K) +1$.
\end{itemize}
\end{thm}
\begin{proof}
By Lemma \ref{lem4.1}, we immediately get that
$$
\mathrm{gl.dim}(K)\leq \mathrm{l.gl.dim}\left(L_K(E)\right)
\text{ and } \mathrm{w.gl.dim}(K)\leq \mathrm{w.gl.dim}\left(L_K(E)\right).$$

We denote by $\{v_1, v_2,\ldots, v_k\}\subseteq E^0$ the set of regular vertices. For each $1\leq i\leq k$, let $E_i$ denote the graph with the same vertices as $E$ and $v_i$ emits the same edges as it does in $E$, but all other vertices do not emit any edge. Then, by Lemma \ref{lem4.2}, we obtain that
$$
\mathrm{l.gl.dim}\left(L_K(E)\right)\leq \max\limits_{1 \le i \le k}\Bigl\{\mathrm{gl.dim}(K) +1,\ \mathrm{l.gl.dim}\left(L_K(E_i)\right)\Bigr\}
$$
and
$$
\mathrm{w.gl.dim}\left(L_K(E)\right)\leq \max\limits_{1 \le i \le k}\Bigl\{\mathrm{w.gl.dim}(K) +1,\ \mathrm{w.gl.dim}\left(L_K(E_i)\right)\Bigr\}.
$$

We next calculate the homological dimensions of $L_K(E_i)$ $(1\leq i \leq k)$. To do so, we consider the following cases:

\emph{Case 1}. $E$ has no any loop $f$ such that $s(f) = v_i$. Then it is easy to see that
$E_i$ is acyclic, and hence, $L_K(E_i)\cong \bigoplus\limits^{k_i}_{j=1}\mathbb{M}_{n_j}(K)$, by Corollary \ref{cor3.5}. Furthermore, since the homological dimensions are Morita invariants, we have that $\mathrm{l.gl.dim}(\mathbb{M}_{n_j}(K)) = \mathrm{gl.dim}(K)$ and $\mathrm{w.gl.dim}(\mathbb{M}_{n_j}(K)) =\mathrm{w.gl.dim}(K)$. These observations imply
$$
\mathrm{l.gl.dim}\left(L_K(E_i)\right) = \mathrm{gl.dim}(K) \text{ and } \mathrm{w.gl.dim}\left(L_K(E_i)\right) =\mathrm{w.gl.dim}(K).
$$

\emph{Case 2}. $E$ contains a loop $f$ having no an exit such that $s(f) = v_i$. We then have that $E_i$ contains a unique cycle $f$, and its other vertices are isolated, so
\[
L_K(E_i)\cong \bigoplus\limits_{j=1}^{n-1}K\oplus K[x, x^{-1}],
\]
where $n = |E^0|$, by Corollary \ref{cor3.5}. Also, by \cite[Theorem 1.7]{douglas:twgdogroar} and \cite[Theorem 2]{dicks:mvpocor}, $\mathrm{gl.dim}\left(K[x, x^{-1}]\right) = \mathrm{gl.dim}(K)+1$ and $\mathrm{w.gl.dim}\left(K[x, x^{-1}]\right) = \mathrm{w.gl.dim}(K)+1$, respectively. From these observations, we immediately get that
$$
\mathrm{l.gl.dim}\left(L_K(E_i)\right) = \mathrm{gl.dim}(K) + 1
$$
and
$$
\mathrm{w.gl.dim}\left(L_K(E_i)\right) =\mathrm{w.gl.dim}(K) + 1.
$$

\emph{Case 3}. $E$ contains a loop $f$ having an exit such that $s(f) = v_i$. Let $F$ be
the subgraph of $E_i$ as follows:
$$
F^0 := \left\{v_i, r(f)\mid f\in s^{-1}(v_i)\right\} \text{ and }
F^1 := s^{-1}(v_i).
$$

We then have that each vertex $w\in E^0\setminus F^0$ is an isolated vertex in $E_i$, and hence,
\[
L_K(E_i) \cong \bigoplus_{w\in E^0\setminus F^0}K \oplus L_K(F).
\]

It implies that $\mathrm{l.gl.dim}\left(L_K(E_i)\right) = \mathrm{l.gl.dim}\left(L_K(F)\right)$ and  $\mathrm{w.gl.dim}\left(L_K(E_i)\right) = \mathrm{w.gl.dim}\left(L_K(F)\right),$ by Lemma \ref{lem4.1}.

By $F$ satisfies Condition (L) and Proposition \ref{prop4.6}, $L_K(F)$ is Morita equivalent to $L_K(F_{v_i})$, so $\mathrm{l.gl.dim}\left(L_K(F)\right) = \mathrm{l.gl.dim}\left(L_K(F_{v_i})\right)$ and  $\mathrm{w.gl.dim}\left(L_K(F)\right) =
\mathrm{w.gl.dim}\left(L_K(F_{v_i})\right)$. Furthermore, it is easy to see that $F_{v_i}$ has no any loop, and $v_i$ and $v^*_i$ are only its regular vertices. Then, applying Lemma \ref{lem4.2} and Case 1, we immediately get that $\mathrm{l.gl.dim}\left(L_K(F_{v_i})\right)\leq \mathrm{gl.dim}(K) + 1$ and $\mathrm{w.gl.dim}\left(L_K(F_{v_i})\right)\leq \mathrm{w.gl.dim}(K) + 1$, and hence,
$$
\mathrm{l.gl.dim}\left(L_K(E_i)\right)\leq \mathrm{gl.dim}(K) + 1
$$
and
$$
\mathrm{w.gl.dim}\left(L_K(E_i)\right)\leq \mathrm{w.gl.dim}(K) + 1.
$$

From these observations, we get the statements, finishing the proof.
\end{proof}

We next consider a criterion for equalities in Theorem \ref{thm4.7} can be achieved. Notice that if $K$ is a commutative ring such that $\mathrm{gl.dim}(K) = \infty$ (resp., $\mathrm{w.gl.dim}(K) \\= \infty$), then we always have that $\mathrm{gl.dim}\left(L_K(E)\right) = \infty$ (resp., $\mathrm{w.gl.dim}\left(L_K(E)\right) = \infty$), by Lemma \ref{lem4.1}. So, we only consider this problem for the case when $K$ is a commutative ring of finite homological dimensions. The following result provides us with an answer to the problem.

\begin{thm}\label{thm4.8}
Let $R$ be a commutative unital algebra over the field $K$ such that $\mathrm{gl.dim}(R) < \infty$, and let
$E$ be a finite graph. Then the following conditions are equivalent:
\begin{itemize}
\item[(1)] $\mathrm{l.gl.dim}(L_R(E)) = \mathrm{gl.dim}(R)$;

\item[(2)] $E$ is acyclic.

\end{itemize}
\end{thm}
\begin{proof}
(1)$\Longrightarrow$(2). Assume that $\mathrm{l.gl.dim}(L_R(E)) = \mathrm{gl.dim}(R)$. We will prove that $E$ is acyclic. Indeed, we first have that $L_R(E) \cong R\otimes_K L_K(E)$ as $R$-algebras, by \cite[Theorem 8.1]{tomf:lpawciacr}. But then, by  \cite[Proposition 10(2)]{erz:otdomaadotp}, which is an immediate consequence of \cite[Theorem XI. 3.1, page 209]{ce:ha}, we get that
$$
\mathrm{l.gl.dim}(L_R(E)) = \mathrm{l.gl.dim}(R\otimes_K L_K(E))\geq \mathrm{gl.dim}(R) + \mathrm{w.gl.dim}(L_K(E)),
$$
so $\mathrm{w.gl.dim}(L_K(E)) = 0$. This implies that $L_K(E)$ is a (von Neumann) regular ring
(see, e.g., \cite[Example 5.62a, page 185]{l:lomar}). Therefore, $E$ is an acyclic graph, by using \cite[Theorem 1]{ar:rcfalpa}.

(2)$\Longrightarrow$(1). Assume that $E$ is an acyclic graph. Then, by Corollary \ref{cor3.5}, $L_R(E)$ is isomorphic to $\bigoplus\limits_{j=1}^k\mathbb{M}_{n_j}(R)$, where $k$ is the number of all sinks (say $\{v_1, \ldots, v_k\}$), and $n_j$ is the number of paths ending in the sink $v_j$.
Furthermore, it is known that $\mathrm{l.gl.dim}\left(\mathbb{M}_{n_j}(R)\right)= \mathrm{gl.dim}(R)$, and hence, $\mathrm{l.gl.dim}(L_R(E))= \mathrm{gl.dim}(R)$, finishing the proof.

\end{proof}

\begin{rem}
~\
\begin{itemize}
\item[(1)] The analogous of Theorem \ref{thm4.8} is equally true for weak global dimension. It is done similarly to the proof of Theorem \ref{thm4.8} and using the fact that if $K$ is a field and $S$ and $T$ are unital $K$-algebras, then $\mathrm{w.gl.dim}(S) + \mathrm{w.gl.dim}(T)\leq \mathrm{w.gl.dim}(S\otimes_kT)$, which is an immediate consequence of \cite[Remark 1]{erz:otdomaadotp}.
\item[(2)] Let $R$ be a commutative ring as in Theorem \ref{thm4.8}, and $E$ a finite graph. Theorem \ref{thm4.8} then shows that
    \begin{equation*}
    \mathrm{l.gl.dim}(L_R(E)) =  \left\{
    \begin{array}{lcl}
    \mathrm{gl.dim}(R),&  & \text{if }E \text{ is acyclic}, \\
    &  &  \\
    \mathrm{gl.dim}(R) + 1,&  & \text{otherwise. \ \ }%
    \end{array}%
    \right.
    \end{equation*}%
    The analogous result holds for weak global dimension.
\item[(3)] In \cite[Theorem 3.5]{amp:nktfga}, the authors proved that if $K$ is a field and $E$ is a finite graph, then $\mathrm{l.gl.dim}(L_K(E))\leq 1$, by using the nice machinery developed by Bergman \cite[Theorem 5.2]{b:casurc}. As an immediate consequence of Theorem \ref{thm4.7}, we may give another proof for this result.
\end{itemize}
\end{rem}

\paragraph{{\bf Acknowledgements.}} The second author is supported by the Vietnam National Foundation for Science and Technology Development (NAFOSTED) under Grant 101.04-2017.19. The authors take an opportunity to express their deep gratitude to Prof. Leonid A. Bokut and Prof. P.N. \'Anh (the Alfr\'ed R\'enyi Mathematical Institute, Hungarian Academy of Sciences) for their valuable suggestions in order to give the final shape of the paper.

\vskip 0.5 cm \vskip 0.5cm {

\end{document}